\documentclass[12pt]{article}
\usepackage{epsf}
\usepackage{amsbsy,amsmath}
\usepackage{mathtools}
\usepackage{mathrsfs}
\usepackage{amsfonts}
\usepackage{amssymb}
\usepackage{enumitem}
\usepackage{eucal}
\usepackage{graphics,mathrsfs}
\usepackage{amsthm}
\usepackage{secdot}
\newtheorem{theorem}{Theorem}[section]
\newtheorem{lemma}[theorem]{Lemma}
\newtheorem{proposition}[theorem]{Proposition}
\newtheorem{corollary}[theorem]{Corollary}
\theoremstyle{definition}
\newtheorem{definition}[theorem]{Definition}

\theoremstyle{remark}
\newtheorem{remark}[theorem]{Remark}
\numberwithin{equation}{section}

\numberwithin{equation}{section}

\newsavebox{\savepar}

\begin{document}
\title{\bf A study of second order semilinear elliptic PDE involving measures}
\author{Ratan Kr Giri \footnote{Corresponding author, Date: \today.} \& Debajyoti Choudhuri   \\
 {\it{\small Department of Mathematics, National Institute of Technology, Rourkela, India }}\\
{\it{\small e-mails: giri90ratan@gmail.com,
dc.iit12@gmail.com}}}

\date{}
\maketitle
\begin{abstract}
\noindent The objective of this article is to study the boundary value problem for the general semilinear elliptic equation of second order involving $L^1$ functions or Radon measures with finite total variation. The study investigates the existence and uniqueness of `{\it very weak}'  solutions to the boundary value problem for a given $L^1$ function. However, a `{\it very weak}'  solution need not exist when an $L^1$ function is replaced with a measure due to which the corresponding reduced limits has been found for which the problem admits a solution in a `{\it very weak}' sense.
\end{abstract}

\section{Introduction and preliminaries}
Solving PDEs with $L^1$ functions or measures as data became very fashionable in the modern theory of PDEs. The motivation for studying such problems have been
discussed beautifully by Brezis in the preface of \cite{brezis3}. One of the most important example where the measure data arise naturally in the nonlinear PDE enters from the heat generation. Heat generation from the exothermic reaction driven by the Arrhenius reaction-term with the pre-exponential factor of the Transition state theory \cite{levine} can be presented by the semilinear elliptic PDE with nonlinear term given by
$$k(u)= c_1 u\exp\left(-\frac{c_2}{u}\right)\,\,\mbox{for}\,\,u>0,$$ where 
$c_1, c_2>0$ are the parameters. Here the function $u$ represents the thermodynamic temperature of this
model. For the analytical treatment, define $k(0) := 0$ and consider an odd extension of the function $k$ by
inserting an absolute value $|u|$, i.e.
$$k(u)=c_1 u\exp\left(-\frac{c_2}{|u|}\right).$$ Then the heat generation can be described by the following PDE involving measure
\begin{align}
\begin{split}
-\Delta u  & = \lambda k(u)+\mu\,\,\mbox{in}\,\,\Omega,\\
u & = 0\,\,\mbox{on}\,\,\partial\Omega. 
\end{split}
\end{align}
We remark that for example heating of the substance at one single point by laser can be expressed by
taking $\mu:=\delta_{x_0}$ being the Dirac measure concentrated at point $x_0\in\Omega$ \cite{atkindepaula}. The PDEs involving measures also have an important role in the theory of probability and in the use of probabilistic methods \cite{dynkin} which gives a new strength to the whole subjects in the recent years. 

In the present article, we are concerned with the boundary value problems for the general second order semilinear elliptic equation involving measures of finite total variation. Problems of this type, involving elliptic operators modeled upon the Laplacian or the $p$-Laplacian, have been systematically studied in the literature, starting with the papers \cite{bocaardogallouet, bocaardogallouet1}, where measure on the right-hand side are considered. Contribution to this topic can be found in \cite{alvinomercaldo}, \cite{alvinoferone}, \cite{Mingionemeasure} and the references therein. In all these articles the elliptic operator which has been considered are either the Laplacian or the $p$-Laplacian. In 2004,  V\'{e}ron\cite{ver} studied the elliptic PDE involving measures where a general linear second order elliptic operator with variable coefficients is appeared, which is precisely the following  \begin{align}
\begin{split}
-L u  & = \lambda\,\,\mbox{in}\,\,\Omega,\\
u & = \mu\,\,\mbox{on}\,\,\partial\Omega,
\end{split}
\end{align}
where $\Omega$ is a smooth domain in $\mathbb{R}^N$, $L$ is a general linear elliptic operator of second order, $\lambda$ and $\mu$ are Radon measures, respectively in $\Omega$ and $\partial\Omega$. Motivated by the interest shared by the mathematical community in this topic, we
study here the existence and uniqueness of solutions to the following Dirichlet problem of the form 
\begin{align}
\begin{split}
-L u + g\circ u & = \mu\,\,\mbox{in}\,\,\Omega,\\
u & = \nu\,\,\mbox{on}\,\,\partial\Omega, \label{ineq1}
\end{split}
\end{align}
where, $\Omega$ is a bounded domain in $\mathbb{R}^N$ with $C^2$ boundary $\partial\Omega$, $L$ is a linear second order differential operator in divergence form, given by 
\begin{equation}
Lu=\sum_{i,j=1}^{N}\frac{\partial}{\partial
	x_i}\left(a_{ij}(x)\frac{\partial u}{\partial
	x_j}\right)-\sum_{j=1}^{N}b_j(x)\frac{\partial u}{\partial
	x_j}+\sum_{j=1}^{N}\frac{\partial (c_j(x) u)}{\partial x_j}-du,\label{Ldef}
\end{equation}
where the functions $a_{ij}$, $b_j$, $c_j$ and $d$ are Lipschitz continuous in $\Omega$
and the principle part of $L$ satisfies the uniform ellipticity condition,
\begin{equation}
\sum_{i,j=1}^{N}a_{ij}(x)\xi_i \xi_j\geq \alpha \sum_{i=1}^N
\xi_i^2, \,\,\forall \xi= (\xi_1,
\xi_2,\cdots,\xi_N)\in\mathbb{R}^N\label{ellip}
\end{equation}
for almost all $x\in\Omega$ with $\alpha>0$ and the input data $\mu, \nu$ are supposed to be Radon measures over $\Omega$, $\partial\Omega$ respectively and $g$ is a given nonlinear function defined on $\Omega \times \mathbb{R}$ with
$g\circ u(x) = g(x, u(x))$. We also assume the following conditions
on $g$:
\begin{align}
\begin{split}
 & \text{(a)}\text{ $g(x, \cdot) \in C(\mathbb{R}),\,\,\,\,g(x, 0)=0,$ } \\
& \text{(b)}\text{ $g(x, \cdot)$ is non decreasing},\\
& \text{(c)} \text{ $g(\cdot, t)\in L^1(\Omega,
\rho)\,,\,\,$}\label{ineq2}
\end{split}
\end{align}
where $L^1(\Omega, \rho)$ denotes the weighted Lebesgue space with
the weight $\rho(x) = \mbox{dist}(x,\partial\Omega)$ for $x\in\bar{\Omega}$. The family of
functions satisfying $(\ref{ineq2})$, will be denoted by
$\mathscr{G}_0$. Observe that if $g\in\mathscr{G}_0$, then the
function $g^*$ given by $g^*(x, t)= -g(x, -t)$ is also in
$\mathscr{G}_0$. Some examples of the nonlinear function $g(x, u(x))$ are the following:
$|u|^q$ for $q\geq 1$, $e^{au}-1$ where $a>0$, $e^{-k/\rho}|u|^{q-1}u$ where $k\geq 0$ \& $q>1$, $\rho(x)^\alpha|u|^q\text{sign}(u)$ where $\alpha>-2$ \& $q>1$, $\rho(x)^\alpha|u|^{q-1}u$ for $q>1$ etc.\\
If $L$ is defined by ($\ref{Ldef}$), then its adjoint operator $L^*$ is given by
\begin{eqnarray}
L^{*}\varphi &=& \sum_{i,j=1}^{N}\frac{\partial}{\partial
	x_j}\left(a_{ij}\frac{\partial \varphi}{\partial
	x_i}\right)-\sum_{j=1}^{N}c_j\frac{\partial \varphi}{\partial
	x_j}+\sum_{j=1}^{N}\frac{\partial}{\partial
	x_j}(b_j\varphi)-d\varphi\label{adj}
\end{eqnarray}
We assume an important uniqueness condition, symmetric in the  $b_j$ and $c_j$, is the following
\begin{equation}
\int_{\Omega} \left( dv +\sum_{j=1}^N\frac{1}{2}(b_j+c_j)\frac{\partial v}{\partial x_j}\right)dx \geq 0,~~\forall~v\in C_c^1(\Omega), v\geq 0.\label{unieque} 
\end{equation}
Under the assumption that the coefficients $a_{ij}$, $b_j$, $c_j$ and $d$ are bounded and measurable in $\Omega$, the uniform ellipticity condition ($\ref{ellip}$), and the uniqueness condition ($\ref{unieque}$), the two operators $L$ and $L^*$ define an isomorphism between $W_0^{1,2}(\Omega)$ and $W^{-1,2}(\Omega)$. Through out this paper, we assume for the operator $L$, the functions $a_{ij}$, $b_j$, $c_j$ and $d$ are Lipschitz continuous functions in $\Omega$, the uniform ellipticity condition $(\ref{ellip})$ and the uniqueness condition ($\ref{unieque}$) holds.

\par Not many evidences are
found in the literature which addresses the problem of existence of
a solution to the equation $(\ref{ineq1})$ with measure data and
hence the reader is suggested to refer to Brezis \cite{brez} which
is one of the earliest attempts made in studying the non-linear
equations with measure data. In fact, he considered the equation of
the type
\begin{align}
\begin{split}
-\Delta u + |u|^{p-1}u & = f(x)\,\,\mbox{in}\,\,\Omega,\\
u & = 0\,\,\mbox{on}\,\,\partial\Omega,
\end{split}
\end{align}
where $\Omega\subset \mathbb{R}^N$ and $0\in \Omega$
with $f$ a given function in $L^1(\Omega)$ or a measure. 
A detailed study of non-linear elliptic partial differential
equations of the above type with measures can be found in Brezis et
al \cite{brezis2}.
Here they have introduced the notion of {\it `reduced limit'}. Readers will perhaps often need to refer to Marcus and V\'{e}ron
\cite{veron} for its richness in addressing problems concerning the
existence of a solution to the nonlinear, second order elliptic
equations involving measures. Some other pioneering contribution to
nonlinear problems with $L^1$ data or measure data which is worth
mentioning are due to Brezis $\&$ Strauss \cite{brezis1}, Marcus
$\&$ Ponce \cite{marcus}, Bhakta and Marcus \cite{bhakta} and the references
therein. The present work in this article
draws its motivation from Marcus
$\&$ Ponce \cite{marcus} and Bhakta and Marcus \cite{bhakta} in which they have
considered the problem ($\ref{ineq1}$) for $L=\Delta$, with data $(\mu, 0) $ and $(0, \nu)$ respectively. In this article we address the problem
for a general linear, second order, elliptic differential operator $L$ and also with input data $(\mu, \nu)$. For an general elliptic operator $L$, things become more complicated if the associated adjoint
is not self adjoint.

\par We now begin our approach to the problem $(\ref{ineq1})$ by defining some of the notations and the definitions which will be quintessential
to our study. We denote $\mathfrak{M}(\Omega)$ to be the space of finite Borel
measures endowed with the norm
$||\mu||_{\mathfrak{M}(\Omega)}=\int_\Omega d|\mu|.$ The measure space $\mathfrak{M}(\Omega)$ is the dual of
$$C_0(\bar{\Omega}) =\{ f \in C(\bar{\Omega}) : f =0
\,\,\mbox{on}\,\, \partial \Omega\}.$$ Similarly, we denote $\mathfrak{M}(\partial\Omega)$ to be the space of bounded Borel measures on $\partial\Omega$ with the usual total variation norm. 
\begin{definition}
	Let $\{\mu_n\}$ be a bounded sequence of measures in
	$\mathfrak{M}(\Omega)$. We say that $\{\mu_n\}$ converges weakly in
	$\Omega$ to a measure $\tau \in \mathfrak{M}(\Omega)$ if $\{\mu_n\}$
	converges weakly to $\tau$ in $\mathfrak{M}(\Omega)$, i.e.
	$$\displaystyle{\int_{\Omega}\varphi d\mu_n} \rightarrow
	\int_{\Omega}\varphi d\tau\,\,; \,\,\,\forall\,\,\varphi\in
	C_0(\bar{\Omega}).$$ We denote this convergence by $\mu_n
	\xrightharpoonup[\Omega]{} \tau$. 
\end{definition}
\noindent We denote by $\mathfrak{M}(\Omega,\rho)$, the
space of signed Radon measures $\mu$ in $\Omega$ such that $\rho\mu
\in \mathfrak{M}(\Omega)$. The norm of a measure $\mu \in
\mathfrak{M}(\Omega, \rho)$ is given by $||\mu||_{\Omega, \rho} =
\int_\Omega \rho d|\mu|.$ This space is the dual of
$$C_0(\bar{\Omega},\rho) = \left\{ h\in C_0(\bar{\Omega}): \frac{h}{\rho}\in
C_0(\bar{\Omega})\right\},$$ where $\frac{h}{\rho} \in
C_0(\bar{\Omega})$ means $\frac{h}{\rho}$ has a continuous extension
to $\bar{\Omega}$, which is zero on $\partial\Omega$.

\begin{definition}
	A sequence $\{\mu_n\}$ in $\mathfrak{M}(\Omega, \rho)$ converges {\it
		`weakly'} to $\mu \in \mathfrak{M}(\Omega, \rho)$ if
	$$\int_\Omega f d\mu_n \rightarrow \int_\Omega f d\mu\, ;\,\,\,
	\forall\,\, f\in C_0(\bar{\Omega},\rho).$$ 
\end{definition}
\noindent The weak convergence in
this sense is equivalent to the weak convergence $\rho\mu_n
\rightharpoonup \rho\mu$ in $\mathfrak{M}(\Omega)$. For this and other properties of weak convergence of
measures we refer to the textbook \cite{veron}. In this article, we
consider the problem $(\ref{ineq1}$) with $\mu\in
\mathfrak{M}(\Omega, \rho)$ and $\nu \in \mathfrak{M}(\partial \Omega)$.\\ The following two definitions of convergence are due to Bhakta and Marcus \cite{bhakta} which are relevant to our study.
\begin{definition}{\label{df1}}
Let $\{\mu_n\}$ be a bounded sequence of measures in
$\mathfrak{M}(\Omega,\rho)$ and $\rho\mu_n$ is extended to a Borel
measure $(\mu_n)_\rho \in \mathfrak{M}(\bar{\Omega})$ defined as zero on
$\partial\Omega$. We say that $\{\rho\mu_n\}$ converge weakly in
$\bar{\Omega}$ to a measure $\tau \in \mathfrak{M}(\bar{\Omega})$ if
$\{(\mu_n)_\rho\}$ converges weakly to $\tau$ in
$\mathfrak{M}(\bar{\Omega})$, i.e.
$$\displaystyle{\int_{\Omega}\varphi\rho d\mu_n} \rightarrow
\int_{\bar{\Omega}}\varphi d\tau\,\,; \,\,\,\forall\,\,\varphi\in
C(\bar{\Omega}).$$ We denote this convergence by $\rho\mu_n
\xrightharpoonup[\bar{\Omega}]{} \tau$.
\end{definition}
\begin{definition}
Let $\{\mu_n\}$ be a sequence in
$\mathfrak{M}_{\text{loc}}(\Omega)$, the space of measures $\mu$ on $$\mathfrak{B}_c=\{E \Subset \Omega: E~\mbox{Borel}\} ,$$ 
such that $\mu\chi_K$ is a finite measure for every compact subset $K \subset \Omega$. We say that $\{\mu_n\}$
converges weakly to $\mu \in \mathfrak{M}_{\text{loc}}(\Omega)$ if
it convergence in the sense of distribution, i.e.
$$\displaystyle{\int_{\Omega}\varphi d\mu_n \rightarrow
\int_{\Omega}\varphi d\mu}\,\,;\,\,\, \forall\,\,\varphi\in C_c(\Omega).$$
We denote this convergence by  $\mu_n\xrightharpoonup[d]{}\mu$.
\end{definition}
\begin{remark}{\label{rk1}}
It can be seen that if $\rho\mu_n\xrightharpoonup[\bar{\Omega}]{}\tau$ then $\mu_n\xrightharpoonup[d]{}\mu_{int}:=\frac{\tau}{\rho}\chi_{\Omega}$. Thus $\tau$ as in the definition $\ref{df1}$,  $\tau=\tau\chi_{\partial\Omega}+\rho\mu_{int}$.
\end{remark}
Let us now come back to our considered semilinear elliptic boundary problem involving measures. Here we will study the existence and uniqueness of {\it `very weak solution'} for the problem $(\ref{ineq1})$. The main reason for attempting the very weak solution instead of weak solution for the problem $(\ref{ineq1})$ comes from the following fact. There are many simple linear elliptic PDEs  of second order with $L^1$ data or measure data on smooth domain $\Omega\subset \mathbb{R}^N$ for which very weak  solutions exists but not weak solutions. For example consider Brezis' problem \cite{brecab}, i.e. Poission equations $-\Delta u =f$ in $\Omega$, under the homogeneous Dirichlet boundary conditions $u=0$ on $\partial\Omega$ for a right hand side $f\in L^1(\Omega, \rho)$. In this Poission problem, for every $f\in L^1(\Omega, \rho)$, existence and uniqueness of a very weak solution $u\in L^1(\Omega)$ satisfying 
$$-\int_{\Omega}u\Delta vdx=\int_{\Omega}fvdx$$ for all $v\in C{^2}(\bar{\Omega})$ with $v=0$ on $\partial\Omega$ is known, but there exists smooth domain $\Omega$ and right hand side function $f\in L^1(\Omega, \rho)$, $f\notin L^1(\Omega)$ such that very weak solution $u$ does not have a weak derivative $\nabla u\in L^1(\Omega)$, i.e. $u\notin W^{1,1}(\Omega)$ and hence is not a weak solution. Thus such a weakening the notion of strong solution is necessary for our considered problem.  
\begin{definition}{\label{def}}
We will define $u\in L^1(\Omega)$ to be a {\it `very weak solution'}
of the problem $(\ref{ineq1})$, if $g\circ u \in L^1(\Omega, \rho)$
and $u$ satisfies the following
\begin{align}
\begin{split}
\int_{\Omega}(- u L^{*} \varphi + (g\circ u)\varphi )dx &= \int_\Omega
\varphi d\mu - \int_{\partial\Omega} \frac{\partial \varphi}{\partial
{\bf{n}}_{L^{*}}} d\nu\,,\,\,\,\forall\,\varphi\in
C^{2,L}_c(\bar{\Omega})\label{ineq3}
\end{split}
\end{align}
where $$C^{2,L}_c(\bar{\Omega}) := \{\varphi\in C^2(\bar{\Omega}) :
\varphi =0\,\,\mbox{on}\,\,
\partial\Omega\,\,\mbox{and}\,\,L^{*}\varphi\in L^{\infty}(\Omega)\}.$$
and  $\frac{\partial \varphi}{\partial
{\bf{n}}_{L^{*}}}=\sum_{i,j=1}^{N}a_{ij}\frac{\partial\varphi}{\partial
x_i}{\bf{n}}_j$, ${\bf{n}}_j$'s are being the component of the outward normal unit vector $\bf{n}$ to
$\partial\Omega$.
\end{definition}
Notice that the co-normal derivative on the boundary following $L^*$, $\frac{\partial \varphi}{\partial {\bf{n}}_{L^{*}}}$ can be written as 
\begin{equation}
\frac{\partial \varphi}{\partial n_{L^{*}}} = \nabla \varphi A \cdot {\bf{n}}= \nabla \varphi \cdot {\bf{n}} A^T\label{conormal}
\end{equation}
where the matrix $A$ is given by $A= (a_{ij})_{N\times N}$ which corresponds to the principle part of the elliptic differential operator $L$. By the uniform ellipticity condition ($\ref{ellip}$), we have ${\bf{n}}\cdot {\bf{n}}A^T>0$. \\
 The most important thing here is that the problem may or may not posses a solution in the very weak sense for every measure. Such an example can be found in Brezis \cite{brez}.  Hence the concept of a {\it{ `good measure'}} was introduced in the literature, which is
defined as follows.
\begin{definition}
We denote by $\mathfrak{M}^g(\bar{\Omega})$ the set of pairs of measures $(\mu, \nu) \in
\mathfrak{M}(\Omega, \rho) \times \mathfrak{M}(\partial\Omega)$ for which the boundary value problem ($\ref{ineq1}$) possesses a solution in very weak sense. If $(\mu, \nu) \in \mathfrak{M}^g(\bar{\Omega})$, we call $(\mu, \nu)$ is a pair of good measures.

\end{definition}
\subsection{Reduced limit}  
Let $\{\mu_n\}$ and $\{\nu_n\}$ be sequences of measures in
$\mathfrak{M}(\Omega, \rho)$ and $\mathfrak{M}(\partial\Omega)$
respectively. Assume that there exists a solution $u_n$ of the
problem ($\ref{ineq1}$) with data $(\mu_n, \nu_n)$, i.e. $u_n$
satisfies the equation ($\ref{ineq3}$) with $\mu= \mu_n$ and $\nu=\nu_n$. Further assume that the sequences of
measures converge in a weak sense to $\mu$ and $\nu$ respectively
while the sequence of very weak solutions $\{u_n\}$ converges to $u$ in $L^1(\Omega)$. In general $u$ is not a very weak solution to the boundary value problem ($\ref{ineq1}$) with data $(\mu, \nu)$. However if there exists measures $(\mu^{\#}, \nu^{\#})$ such that $u$ is a very weak solution of the boundary value problem
($\ref{ineq1}$) with this data, then the pair $(\mu^{\#}, \nu^{\#})$
is called the `{\it reduced limit}' of the sequence $\{\mu_n,
\nu_n\}$. The notion of {\it `reduced limit'} was introduced by Brezis et al. \cite{brezis2} for $L=-\Delta$.  The {\it `reduced measure'} as defined by Brezis et al
\cite{brezis2} is the largest good measure $\leq \mu$ for a Laplacian. In short, the job of a reduced limit of a sequence of measures is to characterize
the class of measures to which the problem has a solution. Here in this work, our main aim is to determined the reduced limit corresponding to our problem ($\ref{ineq1}$).
\par We will use here a well known variational technique
to show existence of solution in $W_0^{1,2}(\Omega)= \{v\in
L^2(\Omega): \nabla v\in L^2(\Omega), v|_{\partial\Omega}=0\}$ with
the Sobolev Norm $||v||_{1,2}= \left(\int_\Omega |\nabla v|^2
dx\right)^{\frac{1}{2}}$. Now let us define $\displaystyle{<u, v> = \int_{\Omega}\sum_{i,j=1}^n
a_{ij}u_{x_i}v_{x_j}}$ over $W_0^{1,2}(\Omega)$. Then the uniform
ellipticity condition $(\ref{ellip})$, implies that $< , >$ is an
inner product on $W_0^{1,2}(\Omega)$. It can be seen that the norm
$||u||= <u, u>^{1/2}$ is equivalent to the Sobolev norm of
$W_0^{1,2}(\Omega)$. This norm equivalence will be effectively used
in the manuscript. The manuscript has been organized into three sections. In Section 2, we begin by studying the
semilinear boundary value problem with $L^1$ data and show certain
basic lemmas and existence theorems. In Section 3, we continue the study by
considering the semilinear problem with measure data and determines the reduced limit corresponding to the problem.
\section{Semilinear problem with $L^1$ data}
In this section we consider the nonlinear boundary value problem 
with $L^1$ data which is as
\noindent follows
\begin{align}
\begin{split}
-L u + g\circ u & =  f \,\,\mbox{in}\,\, \Omega, \\
u & =  \eta \,\,\mbox{on}\,\, \partial\Omega. \label{eq1}
\end{split}
\end{align}
Here $g\in \mathscr{G}_0$, $f\in L^1(\Omega,\rho)$ and $\eta \in
L^1(\partial \Omega)$.\\
Now we have the following result due to Theorem 2.4, \cite{ver}.
\begin{lemma}{\label{lem}}
	Let $f\in L^1(\Omega, \rho)$ and $\eta \in
	L^1(\partial \Omega)$. Then there exists a unique very weak solution  $u\in L^1(\Omega)$ to the problem
	\begin{align}
	\begin{split}
	-L u & =  f \,\,\mbox{in}\,\, \Omega, \\
	u & =  \eta \,\,\mbox{on}\,\, \partial\Omega. \label{e1}
	\end{split}
	\end{align}
	Furthermore, for any $\varphi\in C^{2,L}_c(\bar{\Omega})$, $\varphi\geq 0$, there holds
	$$-\int_\Omega u_+ L^* \varphi dx \leq \int_\Omega f (sign_+ u)\varphi dx - \int_{\partial\Omega} \frac{\partial
		\varphi}{\partial {\bf{n}}_{L^{*}}} d\nu_+$$ and
	$$-\int_\Omega |u| L^* \varphi dx \leq \int_\Omega f
	(sign\, u)\varphi dx - \int_{\partial\Omega} \frac{\partial
		\varphi}{\partial {\bf{n}}_{L^{*}}} d|\nu|.$$
\end{lemma}
\begin{lemma}{\label{lem3}}
	If $u_i\in L^1(\Omega)$ are very weak solutions of $(\ref{eq1})$
	corresponding to $f = f_i$, $\eta = \eta_i$ for $i = 1,2$; then we
	have the following estimate
	\begin{align}
	\begin{split}
	\|u_1-u_2\|_{L^1(\Omega)} + & \|g\circ u_1 - g\circ
	u_2\|_{L^1(\Omega,\rho)} \leq \\
	& C( \|f_1-f_2\|_{L^1(\Omega,\rho)} + \|\eta_1
	-\eta_2\|_{L^1(\partial \Omega)})\label{eq2}
	\end{split}
	\end{align}
	for some $C>0$.
\end{lemma}
\begin{proof}
	Since $u_1, u_2$ are very weak solutions of $(\ref{eq1})$, then we
	have
	\begin{eqnarray*}
		-\int_\Omega u_i L^{*} \varphi dx + \int_\Omega (g\circ u_i)\varphi dx =
		\int_\Omega f_i \varphi dx - \int_{\partial\Omega} \frac{\partial
			\varphi}{\partial {\bf{n}}_{L^{*}}} d\eta_i
	\end{eqnarray*}
	for all $\varphi \in C_c^{2,L}(\bar\Omega)$, $i=1, 2$. Consequently,
	\begin{eqnarray*}
		-\int_\Omega (u_1-u_2) L^{*} \varphi dx + \int_\Omega (g\circ u_1-
		g\circ u_2)\varphi dx = \int_\Omega (f_1-f_2) \varphi dx -
		\int_{\partial\Omega} \frac{\partial \varphi}{\partial {\bf{n}}_{L^{*}}}
		d(\eta_1-\eta_2)
	\end{eqnarray*}
	for all $\varphi \in C_c^{2,L}(\bar\Omega)$. This implies that
	$u_1-u_2$ is a very weak solution of
	\begin{align}
	\begin{split}
	-L u & = f_1 -f_2 - g\circ u_1 + g\circ u_2 \,\,\,\,\,\mbox{in}\,\, \Omega, \\
	u & =  \eta_1 -\eta_2 \,\,\,\,\,\mbox{on}\,\, \partial\Omega.
	\label{eq3}
	\end{split}
	\end{align}
	Therefore, by Lemma $\ref{lem}$, for any $\varphi\in
	C_c^{2,L}(\bar{\Omega})$, $\varphi\geq 0$
	\begin{align}
	-\int_\Omega |u_1-u_2| L^{*} \varphi dx \leq \int_\Omega
	(f_1-f_2-g\circ u_1 & + g\circ
	u_2)sign(u_1-u_2)\varphi dx \nonumber\\
	& - \int_{\partial\Omega} \frac{\partial \varphi}{\partial {\bf{n}}_{L^{*}}}
	d|\eta_1-\eta_2|\label{eq4}
	\end{align}
	Let $\varphi_0$ be the test function satisfying
	\begin{align}
	\begin{split}
	-L^{*}\varphi & = 1 \,\,\,\, \mbox{in}\,\,\Omega,\\
	\varphi & = 0 \,\,\,\, \mbox{on}\,\,\partial\Omega.\label{eq5}
	\end{split}
	\end{align}
	Existence of solution of the PDE $(\ref{eq5})$ is guaranteed by the
	Lemma 2.1 in \cite{ver}. Since the coefficients of $L$ are Lipschtiz continuous,
	from \cite{ver} we have $\varphi_0\in C_c^{2}(\bar\Omega)$ and $L^*\varphi_0\in
	L^\infty(\Omega)$, hence $\varphi_0\in C_c^{2,L}(\bar\Omega)$. It can be seen that $\varphi_0>0$ in
	$\Omega$. This is due to a result in Theorem 2.11, \cite{ver}
	that there exists $\lambda>0$ such that $0<\lambda G_{-\Delta}^{\Omega}<
	G_{L^{*}}^{\Omega}<\lambda^{-1}G_{-\Delta}^{\Omega}$, where
	$G_{-\Delta}^{\Omega}$, $G_{L^{*}}^{\Omega}$ in
	$\Omega\times\Omega\setminus D_{\Omega}$ are the Green's function of
	$-\Delta$, $L^{*}$ respectively and
	$D_{\Omega}=\{(x,x):x\in\Omega\}$. Since from Zhao \cite{zhao},
	$G_{-\Delta}^{\Omega}(x,y)>0$, hence we have
	$\int_{\Omega}G_{-\Delta}^{\Omega}(x,y)dy>0$. It is easy to see 
	that the integral $\int_{\Omega}G_{-\Delta}^{\Omega}(x,y)dy$ is
	finite. Further, there exists $c>0$ such that $c^{-1}<\frac{\varphi_0}{\rho}<c$ in $\Omega$ since $\frac{\varphi_0}{\rho}$ can be continuously extended to $\partial\Omega$ as $\displaystyle{\frac{\varphi_0}{\rho}|_{\partial\Omega}=-\frac{\partial\varphi_0}{\partial{\bf{n}}_{L^*}}\cdot\frac{1}{{\bf{n}}\cdot {\bf{n}}A^T}}$ (refer Corollary $\ref{lem1}$) where $-\frac{\partial\varphi_0}{\partial{\bf{n}}_{L^*}}$ is bounded by the Hopf's lemma (refer Theorem 2.13, \cite{ver}) and $\displaystyle{\frac{1}{{\bf{n}}\cdot {\bf{n}}A^T}}$ is bounded by the uniform ellipticity condition $(\ref{ellip})$. Therefore, taking $\varphi=\varphi_0$ as a test function in $(\ref{eq4})$, we obtain
	\begin{align*}
	\begin{split}
	\int_\Omega |u_1-u_2|dx \leq \int_\Omega
	(f_1-f_2)\varphi_0\text{sign}(u_1-u_2) dx  & + \int_\Omega (-g\circ u_1
	+
	g\circ u_2)sign(u_1-u_2)\varphi_0 dx\\
	& -\int_{\partial\Omega} \frac{\partial \varphi}{\partial {\bf{n}}_{L^{*}}}
	d|\eta_1-\eta_2|
	\end{split}
	\end{align*}
	This implies that
	\begin{align*}
	\|u_1 -u_2\|_{L^1(\Omega)} + \int_\Omega (g\circ u_1  & - g\circ
	u_2)\text{sign}(u_1-u_2)\varphi_0dx \\
	& \leq \int_\Omega (f_1-f_2)\varphi_0\text{sign}(u_1-u_2) dx -
	\int_{\partial\Omega} \frac{\partial \varphi_0}{\partial
		{\bf{n}}_{L^{*}}}d|\eta_1-\eta_2|
	\end{align*}
	By the property of $g$, we have $(g\circ u_1 - g\circ
	u_2)\text{sign}(u_1-u_2)= |g\circ u_1 - g\circ u_2|$. Thus from the
	above equation it follows that
	\begin{align}
	c_{3}(\|u_1 -u_2\|_{L^1(\Omega)} & + \int_\Omega |g\circ u_1 -
	g\circ
	u_2|\cdot\rho dx)\nonumber\\
	&\leq\|u_1-u_2\|_{L^1(\Omega)}  + \int_\Omega |g\circ u_1   - g\circ
	u_2|\cdot\frac{\varphi_0}{\rho}\rho dx \nonumber\\
	& \leq \int_\Omega (f_1-f_2)\varphi_0\text{sign}(u_1-u_2) dx +
	c_0\|\eta_1-\eta_2\|_{L^1(\partial\Omega)}\nonumber\\
	& \leq c\int_\Omega |f_1-f_2|\rho dx +
	c_0\|\eta_1-\eta_2\|_{L^1(\partial\Omega)}\label{eq6}
	\end{align}
	We thus have the result
	\begin{align*}
	\|u_1-u_2\|_{L^1(\Omega)} + & \|g\circ u_1 - g\circ
	u_2\|_{L^1(\Omega; \rho)} \leq \\
	& C( \|f_1-f_2\|_{L^1(\Omega; \rho)} + \|\eta_1
	-\eta_2\|_{L^1(\partial \Omega)}),
	\end{align*}
	if $C$ is chosen to be $\max\{c/c_3,c_0/c_3\}$ where
	$c_3=\min\{1,c^{-1}\}$.\\
	This also implies that if $u \in L^1(\Omega)$ is a very weak solution of
	the boundary value problem $(\ref{eq1})$, then
	\begin{eqnarray}
	\|u\|_{L^1(\Omega)} + \|g\circ u\|_{L^1(\Omega; \rho)} \leq
	C( \|f\|_{L^1(\Omega; \rho)} + \|\eta\|_{L^1(\partial
		\Omega)})\label{eq12}
	\end{eqnarray}
	for some $C>0$.
\end{proof}
\begin{lemma}(\textbf{Comparison of solutions)}{\label{lem4}}
	Let $u_1$ and $u_2$  be very weak solutions in $L^1(\Omega)$ of the
	boundary value problem $(\ref{eq1})$ corresponding to $ f= f_1, \eta=
	\eta_1$ and $f=f_2, \eta=\eta_2$ respectively. If $f_1\leq f_2$ and
	$\eta_1 \leq \eta_2$, then $u_1\leq u_2$ a.e. in $\Omega$.
\end{lemma}
\begin{proof}
	By Lemma $\ref{lem3}$, $u_1 - u_2$ is a weak solution of the problem
	$(\ref{eq3})$. Applying Lemma $\ref{lem}$ with $\varphi = \varphi_0$, where $\varphi_0$ is a solution to $(\ref{eq5})$, we
	have
	\begin{align}
	-\int_\Omega (u_1- u_2)_+L^{*}\varphi_0 \,dx \leq \int_\Omega (f_1-f_2
	& -g\circ u_1 + g\circ u_2)sign_+(u_1-u_2)\varphi_0\, dx\nonumber\\
	& -\int_{\partial\Omega}
	\frac{\partial \varphi_0}{\partial
		{\bf{n}}_{L^{*}}}d(\eta_1-\eta_2)_+.\label{eq10}
	\end{align}
	Since $\eta_1\leq \eta_2$, we have $(\eta_1-\eta_2)_+=0$. Then from
	the equation $(\ref{eq10})$, it follows that
	\begin{align}
	\int_\Omega (u_1- u_2)_+ \,dx \leq \int_\Omega (f_1 &-f_2)sign_+(u_1
	-u_2)\varphi_0\, dx\nonumber\\
	& + \int_\Omega (g\circ u_2 -g\circ u_1)sign_+(u_1-u_2)\varphi_0\,
	dx\label{eq11}
	\end{align}
	Since the test function $\varphi_0>0$ and $f_1\leq f_2$, the
	first integral in right-hand side of $(\ref{eq11})$ is less than or equal to
	zero. 
	Now taking $A= {\Omega\cap\{x\in\Omega: g\circ u_2 -g\circ u_1\geq 0\}}$ and $B= {\Omega\cap\{x\in\Omega: g\circ u_2 - g\circ u_1 <0\}}$, we have
	\begin{eqnarray}
	\int_\Omega (g\circ u_2&-&g\circ
	u_1)sign_+(u_1-u_2)\varphi_0dx\nonumber\\
	&=&\left(\int_A + \int_B\right)[(g\circ u_2 - g\circ u_1)sign_+(u_1-u_2)\varphi_0]dx\nonumber\\
	&=&\int_B(
	g\circ u_2 -
	g\circ u_1)\varphi_0dx\nonumber\\
	&\leq &0.\nonumber
	\end{eqnarray}
	Thus from $(\ref{eq11})$, we get $\displaystyle{\int_{\Omega}(u_1-u_2)_{+}dx\leq 0}$
	which shows that $(u_1-u_2)_{+}=0$. Therefore $u_1 \leq u_2$ a.e. in $\Omega$.
\end{proof}
\begin{theorem}{\label{thm1}}(\textbf{Existence of very weak solution})
	The boundary value problem given by $(\ref{eq1})$ possesses a unique very weak solution $u$ in $L^1(\Omega)$. \end{theorem}
\begin{proof}
	We first prove the existence of weak solution with the test function
	space $W_0^{1,2}(\Omega)$, for the case when $f\in L^\infty(\Omega)$
	and $\eta=0$. Now, for each $n\in \mathbb{N}$, take
	$g_n(x,t)=\min\{g(x,|t|), n\}\text{sign}(g)$ and let $G_n(x, \cdot)$
	be the primitive of $g_n(x, \cdot)$ such that $G_n(x, 0)=0$. Note
	that $G_n$ is a non negative function. $u \in W_0^{1, 2}(\Omega)$ is
	a weak solution of the problem $(\ref{eq1})$ with $g=g_n$ and
	$\eta=0$ if
	$$\int_\Omega a_L(u, v)dx +
	\int_\Omega (g_n\circ u)v\,dx = \int_\Omega fv\,dx,\,\,\,\forall\,v
	\in W_0^{1, 2}(\Omega)$$ where $$\displaystyle{a_L(u, v)= \sum_{i,j=1}^N
		a_{ij}\frac{\partial u}{\partial x_j}\frac{\partial v}{\partial x_i}
		+ \sum_{i=1}^N \left(b_i \frac{\partial u}{\partial x_i}v + c_i
		\frac{\partial v}{\partial x_i}u\right) +duv}.$$
	Let $A_L(u, v)=\displaystyle{\int_\Omega a_L(u, v)dx}$, for all $u, v \in W^{1,2}_0(\Omega)$. Then
	by this definition the bilinear form is continuous on
	$W_0^{1,2}(\Omega)$ and
	$$A_L(v, v)= \int_\Omega \left(\sum_{i,j=1}^N
	a_{ij}\frac{\partial v}{\partial x_j}\frac{\partial v}{\partial x_i}
	+ \frac{1}{2}\sum_{i=1}^N (b_i +c_i) \frac{\partial v^2}{\partial
		x_i}+dv^2\right)dx.$$ By the uniqueness condition ($\ref{unieque}$), we have  $\displaystyle{\int_\Omega \left( dv^2 +\sum_{i=1}^N\frac{1}{2}(b_i
		+c_i) \frac{\partial v^2}{\partial x_i} \right)dx \geq  0}$.
	Thus from the uniform ellipticity condition $(\ref{ellip})$ we have,
	$$A_L(v, v)\geq c_1 \int_\Omega |\nabla v|^2 dx,\,\,\,\forall
	v\in W^{1,2}_0(\Omega).$$ Let us consider the functional
	$$ I_n(u) = A_L(u, u) + \int_\Omega
	(G_n\circ u) \,dx -\int_\Omega fu\,dx$$ over $ W_0^{1, 2}(\Omega)$.
	Since $u \in  W_0^{1, 2}(\Omega)$ and  $W_0^{1,
		2}(\Omega)\hookrightarrow L^2(\Omega)\hookrightarrow L^1(\Omega)$ ,
	we have
	\begin{align*}
	I_n(u) & \geq  c_1 \|\nabla u\|_2^2 + \int_\Omega (G_n\circ u)
	\,dx - \|u\|_{1}\cdot \|f\|_\infty \\
	& \geq c_1 \|\nabla u\|_2^2 + \int_\Omega (G_n\circ u) \,dx - c_2\|\nabla u\|_{2}\cdot \|f\|_\infty \\
	& = \left(c_1 \|\nabla u\|_{2}- c_2\|f\|_\infty\right)\|\nabla
	u\|_{2} + \int_\Omega (G_n\circ u) \,dx,
	\end{align*}
	where $c_1,c_2>0$ are constants. Since $G_n$ is a nonnegative
	function, it shows that $I_n(u)\rightarrow \infty$, when $\|\nabla
	u\|_{2}\rightarrow \infty$. Therefore, the functional $I_n(u)$ is
	coercive.\\
	Now we will show that the functional $I_n(u)$ is weakly lower semi-continuous. For this let $v_m \rightharpoonup u$ weakly in $W_0^{1,
		2}(\Omega)$. By Fatou's lemma,
	$$\int_\Omega G_n\circ u \,dx \leq \lim_{m\rightarrow \infty}\inf
	\int_\Omega G_n\circ v_m\,dx.$$ Now the first term of $A_L(v, v)$ is
	equivalent to the Sobolev norm of $W_0^{1,2}(\Omega)$ and $a_{ij}$'s are Lipschitz continuous functions in $\Omega$ , hence
	\begin{equation}
	\int_{\Omega}\sum_{i,j=1}^N a_{ij}\frac{\partial u}{\partial x_j}\frac{\partial u}{\partial x_i}\leq \lim_{m\rightarrow \infty} \inf \int_{\Omega}\sum_{i,j=1}^N a_{ij}\frac{\partial v_m}{\partial x_j}\frac{\partial v_m}{\partial x_i}.\label{weq1}
	\end{equation}
	Since the embedding $W_0^{1,2}(\Omega) \hookrightarrow L^2(\Omega)$ is compact therefore $v_m \rightarrow u$ in $L^2(\Omega)$ and also we have 
	$\frac{\partial v_m}{\partial x_i}\rightharpoonup \frac{\partial u}{\partial x_i}$ in $L^2(\Omega)$ for each $i=1,2,\cdots,N$. Since $b_i$'s are Lipschitz continuous functions on $\Omega$, by the strong convergence of $v_m$ in $L^2(\Omega)$ and the weak convergence of $\frac{\partial v_m}{\partial x_i}$ in $L^2(\Omega)$, one can see that 
	\begin{eqnarray*}
		\lim_{m\rightarrow\infty} \lim_{n\rightarrow\infty} \int_{\Omega} b_i v_m \frac{\partial v_m}{\partial x_i} dx &=& \lim_{n\rightarrow\infty} \lim_{m\rightarrow\infty} \int_{\Omega} b_i v_m \frac{\partial v_m}{\partial x_i} dx\nonumber\\ &=& \int_{\Omega} b_i u \frac{\partial u}{\partial x_i} dx\nonumber
	\end{eqnarray*}
	Therefore taking $m=n$, we have $\displaystyle{\int_{\Omega} b_i v_m \frac{\partial v_m}{\partial x_i} dx \rightarrow \int_{\Omega} b_i u \frac{\partial u}{\partial x_i} dx}$ as $m \rightarrow \infty$ for each $i=1,2,\cdots, N$. Thus
	\begin{equation}
	\lim_{m\rightarrow\infty}\int_\Omega\frac{1}{2}\sum_{i=1}^N (b_i +c_i) \frac{\partial v_m^2}{\partial x_i} = \int_{\Omega}\frac{1}{2}\sum_{i=1}^N (b_i +c_i) \frac{\partial u^2}{\partial x_i}.\label{weq2}
	\end{equation}
	Similarly, for the third term of $A_L(v,v)$ we have 
	\begin{equation}
	\lim_{m\rightarrow\infty}\int_{\Omega}d v_m^2 = \int_{\Omega} du^2.\label{weq3}
	\end{equation}
	Therefore, combining ($\ref{weq1}$), ($\ref{weq2}$) and ($\ref{weq3}$) we have
	\begin{align*}
	I_n(u) & \leq \lim_{m\rightarrow \infty}\inf A_L(v_m, v_m)+
	\lim_{m\rightarrow \infty}\inf \int_\Omega G_n\circ v_m\,dx - \lim
	_{m\rightarrow
		\infty} \int_\Omega fv_m\\
	&\leq \lim _{m\rightarrow \infty}\inf I_n(v_m).
	\end{align*}
	Thus $I_n(u)$ is weakly lower semi-continuous and coercive. Hence
	the variational problem $\underset{u\in W_0^{1,2}(\Omega)}{\min}
	\{I_n(u)\}$ possesses a weak solution $u_n\in W_0^{1, 2}(\Omega)$.
	The minimizer $u_n$ is a weak solution of the boundary value
	problem
	\begin{align}
	\begin{split}
	-L u + g_n\circ u & =  f \,\,\mbox{in}\,\, \Omega, \\
	u & = 0 \,\,\mbox{on}\,\, \partial\Omega, 
	\end{split}
	\end{align}
	where $f\in L^{\infty}(\Omega)$. That is $u_n \in W_0^{1, 2}(\Omega)$ satisfies, 
	\begin{equation}
	\int_{\Omega}\left(\sum_{i,j=1}^N
	a_{ij}\frac{\partial u_n}{\partial x_j}\frac{\partial v}{\partial x_i}
	+ \sum_{i=1}^N \left(b_i \frac{\partial u_n}{\partial x_i}v + c_i
	\frac{\partial v}{\partial x_i}u_n\right) +du_nv\right) + \int_{\Omega}(g_n\circ u_n)v = \int_{\Omega}fv, \label{wsol}
	\end{equation}
	for every $v\in  W_0^{1, 2}(\Omega)$. Thus by taking $v=\varphi$, where $\varphi \in C_c^{2,L}(\bar{\Omega})$ in the equation ($\ref{wsol}$) and then applying integration by parts we get 
	\begin{equation}
	-\int_{\Omega} u_n L^*\varphi + \int_{\Omega}(g_n\circ u_n)v = \int_{\Omega}f\varphi
	\end{equation} 
	for every $\varphi\in C_c^{2,L}(\bar{\Omega})$. This shows that $u_n$ is a very weak solution of the boundary value problem
	\begin{align}
	\begin{split}
	-L u + g_n\circ u & =  f \,\,\mbox{in}\,\, \Omega, \\
	u & = 0 \,\,\mbox{on}\,\, \partial\Omega, \label{eq13}
	\end{split}
	\end{align}
	where $f\in L^\infty(\Omega)$. Further, by ($\ref{eq12}$), the
	sequences $\{u_n\}$ and $\{g_n\circ u_n\}$ are bounded in $L^1(\Omega)$
	and $L^1(\Omega,\rho)$ respectively. 
	
	\noindent Now consider the case
	$f\geq 0$. Then by comparison of solutions (by the Lemma $\ref{lem4}$) we obtain $u_n\geq 0$. Since $u_n$
	is a very weak solution of the problem $(\ref{eq13})$, we write as following
	\begin{align}
	\begin{split}
	-L u_n + g_n\circ u_n & =  f \,\,\mbox{in}\,\, \Omega, \\
	u_n & = 0 \,\,\mbox{on}\,\, \partial\Omega. \label{eq14}
	\end{split}
	\end{align}
	A slight manipulation of $(\ref{eq14})$ gives the following
	\begin{align}
	\begin{split}
	-L u_n + g_{n+1}\circ u_n & =  f + g_{n+1}\circ u_n - g_n\circ u_n \,\,\mbox{in}\,\, \Omega, \\
	u_n & = 0 \,\,\mbox{on}\,\, \partial\Omega. \label{eq15}
	\end{split}
	\end{align}
	Choose $f^* =f + g_{n+1}\circ u_n - g_n\circ u_n$, then $f^*\geq f$
	on $\Omega$ because the sequence $\{g_n\}$ is monotonically
	increasing. We also have $u_{n+1}$, which is a very weak solution to the
	problem
	\begin{align}
	\begin{split}
	-L u_{n+1} + g_{n+1}\circ u_{n+1} & =  f \,\,\mbox{in}\,\, \Omega, \\
	u_{n+1} & = 0 \,\,\mbox{on}\,\, \partial\Omega. \label{eq16}
	\end{split}
	\end{align}
	Since $f^*\geq f$, hence  from ($\ref{eq15}$) and ($\ref{eq16}$) we have
	$u_{n+1}\leq u_n$. Thus $\{u_n\}$ is a bounded monotonically decreasing sequence and so by the dominated convergence theorem we have $u_n
	\rightarrow u$ in $L^1(\Omega)$, for some $u$. Therefore there
	exists a subsequence, which we will still denote as $u_n$, converges to $u$ pointwise a.e. and hence $g_n\circ u_n \rightarrow g\circ u$. Indeed,
	\begin{align*}
	g_n\circ u_n(x) & = \min \{g(x, |u_n(x)|), n\}\, sign(g)\\
	& = \min \{g(x, u_n(x)), n\}\,sign(g)\\
	& = g(x, u_n(x)),\,\,\,\mbox{for}\,\,n\geq k(x)
	\end{align*}
	From ($\ref{ineq2}$), we have $g_n\circ u_n(x) = g\circ u_n(x)
	\rightarrow g\circ u(x)$ a.e. for $n \geq k(x)$. Now by the Theorem $2.4$ of V\'{e}ron \cite{ver},
	let $V$ be the very weak solution of
	\begin{align}
	\begin{split}
	-L v & = f \,\,\mbox{in}\,\, \Omega,\\
	v & =0 \,\,\mbox{on}\,\, \partial\Omega. \label{eq17}
	\end{split}
	\end{align}
	Notice that as $u_n\geq 0$, we have $g_n\circ u_n \geq 0$. Thus,
	\begin{align*}
	-L u_n & = f- g_n \circ u_n \leq f = -L v \,\,\mbox{in}\,\, \Omega,\\
	u_m & =0\,\,\,\,\,~~~~~~~~~~~~~~~~ v= 0 \,\,\mbox{on}\,\,
	\partial\Omega.
	\end{align*}
	Therefore, by comparison of solutions, we have $u_n\leq V$ and hence
	$g\circ u_n \leq g\circ V$. In other words, if $V$ is a very weak solution of the boundary value problem $(\ref{eq17})$, then the sequence $\{g\circ u_n\}$ is dominated by
	$g\circ V$. Since $u_n \rightarrow u $ and $g\circ u_n\rightarrow
	g\circ u$ in $L^1(\Omega)$, hence $\displaystyle{\int_\Omega u_n
		L^*\varphi \rightarrow \int_\Omega u L^*\varphi}$ and
	$\displaystyle{\int_\Omega (g\circ u_n)\varphi \rightarrow \int_\Omega
		(g\circ u)\varphi}$ for all $\varphi\in C^{2,L}_c(\bar{\Omega})$. Thus we
	can conclude that $u\in L^1(\Omega)$ is a very weak solution of
	\begin{align}
	\begin{split}
	-L u + g\circ u & =  f \,\,\mbox{in}\,\, \Omega, \\
	u & = 0 \,\,\mbox{on}\,\, \partial\Omega. \label{eq18}
	\end{split}
	\end{align}
	\noindent We now drop the condition $f\geq 0$. Let $\tilde{u}_n$ be
	a very weak solution of $(\ref{eq13})$ with $f$ replaced by $|f|$. Then
	$\tilde{u}_n\geq 0$ and
	\begin{align*}
	-L u_n + g_n \circ u_n & = f  \leq |f| = -L \tilde{u}_n + g_n\circ \tilde{u}_n \,\,\mbox{in}\,\, \Omega,\\
	u_n & =0\,\,\,\,\,~~~~~~~~~~~~~~~~~~~~~~~~~~\, \tilde{u}_n= 0
	\,\,\mbox{on}\,\,\partial\Omega.
	\end{align*}
	Hence by the comparison of solutions, we have $u_n\leq \tilde{u}_n$.
	Since $g(x, -\tilde{u}_n(x))\leq 0$, hence one can show that $g_n(x,
	-\tilde{u}_n(x)) = -g_n(x, \tilde{u}_n(x))$ and also
	\begin{align}
	\begin{split}
	-L (-\tilde{u}_n )+ g_n\circ (-\tilde{u}_n )& =  -|f| \,\,\mbox{in}\,\, \Omega, \\
	-(\tilde{u}_n) & = 0 \,\,\mbox{on}\,\, \partial\Omega. \label{eq19}
	\end{split}
	\end{align}
	Again by comparison of solutions we have $-\tilde{u}_n \leq u_n$, as
	$-|f|\leq f$. Therefore, $|u_n|\leq \tilde{u}_n$. By the similar argument as previous, the sequence $\{\tilde{u}_n\}$ is bounded in $L^1(\Omega)$ and monotonically decreasing, hence $\{u_n\}$ is also a bounded monotonically decreasing sequence. Thus $u_n \rightarrow u$ in $L^1(\Omega)$, for some $u$ and therefore there exists a subsequence such that $u_n(x)\rightarrow u(x)$ a.e.. Hence $\{g_n\circ u_n\}$ converges a.e. and is dominated by $\{g_n\circ \tilde{u}_n\}$. Therefore $u$ is a very weak solution of the boundary value problem ($\ref{eq18}$). By using the density arguments in the estimates $(\ref{eq2})$, we obtain the existence
	of very weak solution for every $f\in L^1(\Omega; \rho)$.\\
	Suppose $\eta\neq 0$ and $\eta \in
	C^2(\partial \Omega)$ and let $v$ be a classical solution (refer \cite{ver}) of
	\begin{align}
	\begin{split}
	-L v & =  0 \,\,\mbox{in}\,\, \Omega, \\
	v & = \eta \,\,\mbox{on}\,\, \partial\Omega.\label{eq20}
	\end{split}
	\end{align}
	Let $w= u-v$. So we have $L (w+v)= L w$. Then the problem
	$(\ref{eq1})$ can be written as
	\begin{align}
	\begin{split}
	-L w + \tilde{g}\circ w & =  \tilde{f} \,\,\mbox{in}\,\, \Omega, \\
	w & =  0 \,\,\mbox{on}\,\, \partial\Omega,
	\end{split}
	\end{align}
	where $\tilde{g}\circ w = g(x, w(x)+v(x))- g(x, v(x))$ and
	$\tilde{f} = f - g\circ v$. Clearly $\tilde{g}\in \mathscr{G}_0$ and
	$\tilde{f}\in L^1(\Omega; \rho)$. Therefore the boundary value problem $(\ref{eq1})$ possesses a
	weak solution whenever $f\in  L^1(\Omega,\rho)$ and $\eta \in
	C^2(\partial\Omega)$.\\
	\noindent Suppose $f\in L^1(\Omega, \rho)$ and $\eta \in
	L^1({\partial\Omega})$, by density there exists a sequence $\{\eta_n\}\subset C^\infty(\partial\Omega)$ such that $\eta_n\rightarrow \eta$ in $L^1(\partial\Omega)$. To each ($f, \eta_n$), there exists a very weak solution $u_n\in L^1(\Omega)$. By estimate ($\ref{eq2}$), we
	have $u_n\rightarrow u$ in $L^1(\Omega)$ and $g\circ u_n\rightarrow
	g\circ u$ in $L^1(\Omega,\rho)$. This precisely shows that $u$ is a very weak solution of the boundary value problem ($\ref{eq1}$).
\end{proof}
\section{Semilinear problem with measure data}
In this section we prove the following main result.

\begin{theorem}{\label{thm3}}
Assume that $\{\mu_n,\nu_n\}\subset \mathfrak{M}^g(\bar\Omega)$ such that $\rho\mu_n
\xrightharpoonup[\bar{\Omega}]{} \tau$ in $\mathfrak{M}(\bar{\Omega})$ and
$\nu_n\rightharpoonup\nu$ in $\mathfrak{M}(\partial\Omega)$. Let $u_n$ be the solution of
\begin{align}
\begin{split}
-L u+g\circ u & =\mu_n\,\,\mbox{in}\,\,\Omega\\
u & =\nu_n \,\,\mbox{on}\,\,\partial\Omega 
\end{split}
\end{align}
where $g\in \mathscr{G}_0$ and suppose that $$u_n\rightarrow u\,\,\,\mbox{in}\,\,L^1(\Omega).$$ Then \\
(i) $\{\rho(g\circ u_n)\}$ converges weakly in $\bar{\Omega}$ and \\
(ii) there exists $\mu^{\#} \in \mathfrak{M}(\Omega, \rho)$,
$\nu^{\#}\in \mathfrak{M}(\partial\Omega)$ such that $u$ is a weak
solution of
\begin{align}
\begin{split}
-L u+g\circ u & =\mu^{\#}\,\,\mbox{in}\,\,\Omega\\
u & =\nu^{\#} \,\,\mbox{on}\,\,\partial\Omega .\label{eq39}
\end{split}
\end{align}
Furthermore, if $\mu_n\geq 0$ and $\nu_n\geq 0$ for every $n$, then
$$0\leq \nu^{\#}\leq (\nu +\frac{\tau}{{\bf{n}}\cdot {\bf{n}}A^T} \chi_{\partial\Omega}),$$
where ${\bf{n}}$ is the outward normal unit vector to the boundary $\partial\Omega$ and $A=(a_{ij})_{N\times N}$, the matrix corresponding to the principle part of elliptic differential operator $L$.
\end{theorem}
\noindent The measures $\mu^{\#}$ and $\nu^{\#}$ are called reduced
limit of the sequences of measures $\{\mu_n\}$ and $\{\nu_n\}$
respectively. We divide the proof into several lemmas and theorems.
We now begin with the following existence theorem.
\begin{theorem}{\label{thm6}}
Consider the boundary value problem
\begin{align}
\begin{split}
-L u + g\circ u & =  \mu \,\,\mbox{in}\,\, \Omega \\
u & =  \nu \,\,\mbox{on}\,\, \partial\Omega \label{eq21}
\end{split}
\end{align}
with $g\in \mathscr{G}_0$, $\mu \in \mathfrak{M}(\Omega, \rho)$ and
$\nu \in \mathfrak{M}(\partial \Omega)$. If a solution exists, then 
\begin{equation}
||u||_{L^1(\Omega)}+||g\circ u||_{L^1(\Omega, \rho)}\leq C(||\mu||_{\mathfrak{M}(\Omega, \rho)}+||\nu||_{\mathfrak{M}(\partial\Omega)})\label{estimate}
\end{equation}
If $u_i\in L^1(\Omega)$ are
very weak solutions corresponding to $\mu=\mu_i$, for $i=1,2$, then
we have the following estimate
\begin{eqnarray}
||u_1-u_2||_{L^1(\Omega)} + ||g\circ u_1-g\circ u_2||_{L^1(\Omega;
\rho)} \leq
 C( ||\mu_1-\mu_2||_{\mathfrak{M}(\Omega, \rho)} + ||\nu_1-\nu_2||_{\mathfrak{M}(\partial
 \Omega)})\label{eq22}
 \end{eqnarray}
Furthermore, if $\mu_1\leq \mu_2$, $\nu_1\leq\nu_2$ then $u_1\leq
u_2$. This also implies that the problem in $(\ref{eq21})$ possesses at most one very weak solution $u\in L^1(\Omega)$ if at all a solution exists to it.
\end{theorem}
\begin{proof}
The proof runs along the same lines as that of the corresponding
Lemmas ($2.2$), ($2.3$) and Theorem ($2.4$) in the previous section.
\end{proof}
In contrast to the case of when $L=\Delta$ with $L^1$ data, the
problem with measure data does not necessarily possess a solution. It may so happen that
$\mu_n \rightharpoonup \delta_0$ and $u_n \rightarrow 0$ in $L^1(\Omega)$,
although $0$ is not a solution of ($\ref{eq21}$) with $L=\Delta$,
$\mu=\delta_0$ and $\nu=0$ \cite{brez}. However, if a solution exists then it is unique
and the inequality ($\ref{estimate}$) remain valid.\\
\noindent The following corollary is an immediate consequence of the
definition of a good measure and Theorem $\ref{thm6}$.
\begin{corollary}{\label{thm8}}
Assume that $(\mu, \nu) \in \mathfrak{M}^g(\bar{\Omega})$. Then the
boundary value problem ($\ref{eq21}$) possess a unique very weak solution
in $L^{1}(\Omega)$.
\end{corollary}
We state the following theorem.
\begin{theorem}{\label{thm7}}
Assume that $(\mu, \nu) \in \mathfrak{M}^g(\bar{\Omega})$ with
$\mu\geq 0$ and $\nu\geq 0$. Then the very weak solution $u$ of the boundary value problem ($\ref{eq21}$), is in $L^p(\Omega)$ for $1\leq p <
\frac{N}{N-1}$ and there exists a constant $C(p)$ such that
\begin{align}
||u||_{L^p(\Omega)} \leq C(p)(||\mu||_{\mathfrak{M}(\Omega, \rho)} +
||\nu||_{\mathfrak{M}(\partial\Omega)})\label{eq26}
\end{align}
\end{theorem}
\begin{proof}
The range of $p$ can be found by using the Green function of the
elliptic operator $L$ which is obtained in the work of
V\'{e}ron (\cite{ver}). Note that in our case we are considering $p$
is strictly less than $N$. The estimate is an immediate consequence
of ($\ref{eq22}$) and the notion of representing the solution in terms
of Green's function.
\end{proof}
\begin{corollary}{\label{cor1}}
Under the assumptions made in Theorem $\ref{thm7}$, the solution $u$
of the problem ($\ref{eq21}$) is in $ W^{1,p}_{loc}(\Omega)$ for $ \in
\big[1, \frac{N}{N-1}\big)$. Also for every relatively compact
domain $\Omega '$ in $\Omega$, there exists a constant $C(q)$ such
that
\begin{align}
||u||_{L^p(\Omega ')} \leq C(p)(||\mu||_{\mathfrak{M}(\Omega ')} +
||\nu||_{\mathfrak{M}(\partial\Omega)}).\label{eq27}
\end{align}
\end{corollary}
The following definitions and propositions are due to Marcus and V\'{e}ron \cite{veron}. 

\begin{definition}\label{exhaustion}
	We say that $\{\Omega_n\}$ is uniformly of class $C^2$ if $\exists~r_0, \gamma_0, n_0$ such that for any $X\in\partial\Omega$:\\
	There exists a system of Cartesian coordinates $\xi$ centered at $X$, a sequence $\{f_n\}\subset C^2(B_{r_0}^{N-1}(0))$ and $f\in C^2(B_{r_0}^{N-1}(0))$ such that the following statement holds. Let $$Q_0:=\{\xi=(\xi_1,\xi')\in\mathbb{R}\times\mathbb{R}^{N-1}:|\xi'|<r_0, |\xi_N|<\gamma_0\}.$$ Then the surfaces $\partial\Omega_n\cap Q_0$, $n>n_0$ and $\partial\Omega\cap Q_0$ can be expressed as $\xi_1=f_n(\xi')$ and $\xi_1=f(\xi')$ respectively and $$f_n\rightarrow f~~\text{in}~C^2(B_{r_0}^{N-1}(0)).$$
\end{definition}

\begin{definition}
	A sequence $\{\Omega_n\}$ is an exhaustion of $\Omega$ if
	${\bar{\Omega}_n}\subset\Omega_{n+1}$ and
	$\Omega_n\uparrow\Omega$. We say that an exhaustion $\Omega_n$
	is of class $C^2$ if each domain $\Omega_n$ is of this class. If, in
	addition, $\Omega$ is a $C^2$ domain and the sequence of domains
	$\{\Omega_n\}$ is uniformly of class $C^2$, we say that $\{\Omega_n\}$
	is a uniform $C^2$ exhaustion.
\end{definition}
  
\begin{definition}\label{trdef}
	Let $u \in W_{\text{loc}}^{1,p}(\Omega)$ for some $p>1$. We say that $u$
	possesses an $M$-boundary trace on $\partial\Omega$ if there exists
	$\nu\in\mathfrak{M}(\partial\Omega)$ such that, for every uniform
	$C^2$ exhaustion $\{\Omega_n\}$ and every $h\in C(\bar{\Omega})$,
	\begin{eqnarray}
	\int_{\partial\Omega_n}u\lfloor_{\partial\Omega_n}hdS\rightarrow\int_{\partial\Omega}hd\nu,\nonumber
	\end{eqnarray}
	where $u\lfloor_{\partial\Omega_n}$ denotes the Sobolev trace,
	$dS=d\mathbb{H}^{N-1}$ and $\mathbb{H}^{N-1}$ denotes the $(N-1)$
	dimensional Hausdorff measure. The $M$-boundary trace $\nu$ of $u$ is denoted by $tr\, u$.
\end{definition}
\begin{remark}
	If $u\in W^{1,p}(\Omega)$ for some $p>1$, then the Sobolev trace = $M$- boundary trace.
\end{remark}
\begin{definition}
	We say that $u\in L^1(\Omega)$ satisfies $-Lu =\mu$ in $\Omega$, in the sense of distribution if it satisfies
	$$-\int_{\Omega} uL^*\varphi = \int_{\Omega}\varphi d\mu$$ for every $\varphi \in C^{\infty, L}_c(\Omega)$, where $C^{\infty, L}_c(\Omega) = \{\varphi \in C_c^\infty(\Omega): L^*\varphi \in L^\infty(\Omega)\}$.
\end{definition}
\begin{proposition}\label{prop}
	Let $\mu \in \mathfrak{M}(\Omega, \rho)$ and $\nu \in \mathfrak{M}(\partial\Omega)$. Then a function $u \in L^1(\Omega)$ is a very weak solution of the problem 
	 \begin{align*}
	 \begin{split}
	 -L u & =\mu\,\,\mbox{in}\,\,\Omega\\
	 u & =\nu \,\,\mbox{on}\,\,\partial\Omega 
	 \end{split}
	 \end{align*}
	 if and only if 
	 \begin{align*}
	 \begin{split}
	 -L u & =\mu\,\,\mbox{in}\,\,\Omega\,\,\,(\text{in the sense of distribution})\\
	 tr\,u & =\nu \,\,\mbox{on}\,\,\partial\Omega \,\,\,( \text{in the sense of Definition \ref{trdef}} )
	 \end{split}
	 \end{align*}
	\end{proposition}

	 \begin{proof}
	The proof follows the Proposition 1.3.7, \cite{veron}.
	\end{proof}
The following result is an immediate consequence of the Proposition $\ref{prop}$.
\begin{proposition}
		Let $\mu \in \mathfrak{M}(\Omega, \rho)$ and $\nu \in \mathfrak{M}(\partial\Omega)$. Then a function $u\in L^1(\Omega)$, with $g\circ u\in L^1(\Omega, \rho)$, satisfies $(\ref{ineq3})$ if and only if 
		\begin{align*}
		\begin{split}
		-L u + g\circ u & =\mu\,\,\mbox{in}\,\,\Omega\,\,\,(\text{in the sense of distribution})\\
		tr\,u & =\nu \,\,\mbox{on}\,\,\partial\Omega \,\,\,( \text{in the sense of Definition \ref{trdef}} )
		\end{split}
		\end{align*}
\end{proposition}
\noindent We prove the following crucial lemma.
\begin{lemma}{\label{lem1}}
Let $\rho\mu_n \xrightharpoonup[\bar{\Omega}]{} \tau.$ Then
$$\displaystyle{\lim_{n\rightarrow \infty} \int_{\Omega}\varphi d
\mu_n=\int_{\Omega} \varphi
d\mu_{int}-\int_{\partial\Omega}\frac{\partial \varphi}{\partial
{\bf{n}}_{L^{*}}}\frac{1}{{\bf{n}}\cdot {\bf{n}}A^T} d\tau}$$ for all $\varphi\in C_c^{2,L}(\bar{\Omega})$, where ${\bf{n}}$ is the outward normal unit vector to $\partial\Omega$ and $A=(a_{ij})_{N\times N}$.
\end{lemma}
\begin{proof}
Consider $\varphi \in  C_c^{2,L}(\bar{\Omega})$. Since $\varphi$ vanishes on $\partial\Omega$, so for $x_0\in \partial\Omega$, $\nabla\varphi(x_0)$  is normal to $\partial\Omega$, that is 
$$\nabla \varphi(x_0)= c\,{\bf{n}},~~~~\text{where}~~c:=\frac{\partial\varphi}{\partial {\bf{n}}}(x_0).$$
As $\rho(x)= dist(x, \partial\Omega)$, hence $\nabla\rho(x_0)=-{\bf{n}}$. Thus for given any direction $v$, we have 
$$\lim_{t\rightarrow 0+}\frac{\varphi(x_0-tv)}{\rho(x_0-tv)}= \frac{\nabla\varphi(x_0)\cdot v}{\nabla \rho(x_0)\cdot v} =\frac{\nabla\varphi(x_0)\cdot v}{-{\bf{n}}\cdot v} = -c=-\frac{\partial\varphi}{\partial {\bf{n}}}(x_0).$$
In particular, taking $v={\bf{n}}A^T(x_0)$ in the above we get,
$$\lim_{t\rightarrow 0+}\frac{\varphi(x_0-tv)}{\rho(x_0-tv)} =\frac{\nabla\varphi(x_0)\cdot {\bf{n}}A^T(x_0)}{\nabla \rho(x_0)\cdot {\bf{n}}A^(x_0)} =\frac{\nabla\varphi(x_0)\cdot {\bf{n}}A^T(x_0)}{-{\bf{n}}\cdot {\bf{n}}A^T(x_0)}=-\frac{\partial \varphi}{\partial {\bf{n}}_{L^{*}}}(x_0)\frac{1}{{\bf{n}}\cdot {\bf{n}}A^T(x_0)}.$$
Now take $\bar{\varphi}(x)=\begin{cases}
\frac{\varphi}{\rho}(x)\,; & x\in \Omega\,,\\
-\frac{\partial \varphi}{\partial {\bf{n}}_{L^{*}}}\frac{1}{{\bf{n}}\cdot {\bf{n}}A^T}\,; & x\in
\partial\Omega
\end{cases} $\\
Then $\bar{\varphi}\in C(\bar{\Omega})$ and using remark
$\ref{rk1}$, we have,
\begin{eqnarray*}
\lim_{n\rightarrow \infty} \int_{\Omega}\varphi d \mu_n &=& \lim_{n\rightarrow \infty} \int_{\Omega}\rho \bar{\varphi} d \mu_n\\
&=&\int_{\bar{\Omega}}\bar{\varphi}d\tau\,\,\,\,\,\,\,\,(\text{since}~ \rho\mu_n \xrightharpoonup[\bar{\Omega}]{} \tau)\\
&=&\int_{\bar{\Omega}}\bar{\varphi}\chi_{\partial\Omega} d\tau+\int_{\bar{\Omega}}\rho\bar{\varphi}d\mu_{int} \,\,\,\,\,\,\,\,
(\text{since}~ \tau=\rho\mu_{int}+\tau\chi_{\partial\Omega})\\
&=&\int_{\partial\Omega}\bar{\varphi} d\tau+\int_{\Omega}\rho \bar{\varphi} d\mu_{int}\\
&=&\int_{\Omega}\varphi
d\mu_{int}-\int_{\partial\Omega}\frac{\partial \varphi}{\partial
n_{L^{*}}}\frac{1}{{\bf{n}}\cdot {\bf{n}}A^T} d\tau
\end{eqnarray*}
\end{proof}
\begin{lemma}
	Assume that the given conditions in the Theorem $\ref{thm3}$ are holds. Then there exists a subsequence $\{u_{n_k}\}$ of $\{u_n\}$ that converges
	in $L^1(\Omega)$.
\end{lemma}
\begin{proof}
	By the given condition we have $\|\mu_n\|_{\mathfrak{M}(\Omega, \rho)}+\|\nu_n
	\|_{\mathfrak{M}(\partial \Omega)}\leq c\,,\,\,\forall\, n\in \mathbb{N}$, for some $c>0$.
	Therefore, by ($\ref{eq26}$),  $\{u_n\}$ is bounded in ${L^p(\Omega)}$
	for every $p\in \big[1, \frac{N}{N-1}\big)$. This implies that
	$\{u_n\}$ is uniformly integrable in $L^p(\Omega)$, for each such
	$p$. By Vitali's convergence theorem there exists a subsequence $\{u_{n_k}\}$ such that $u_{n_k}\rightarrow u$ in $L^1(\Omega)$, for some $u\in L^1(\Omega)$.
\end{proof}
\begin{proof}[\textbf{Proof of the Theorem $\ref{thm3}$}]
By our assumption, $\{\mu_n\}$ is bounded in $\mathfrak{M}(\Omega,
\rho)$ and $\{\nu_n\}$ is bounded in $\mathfrak{M}(\partial
\Omega)$. Using ($\ref{estimate}$), we have $\{g\circ u_n\}$ is bounded in
$L^1(\Omega, \rho)$ and hence $\{\rho (g\circ u_n)\}$ is also
bounded $L^1(\Omega)$. Therefore, there exists a subsequence of $\{\rho(g\circ
u_n)\}$ still denoted by $\{\rho(g\circ u_n)\}$ converges weakly in
$\bar{\Omega}$. Thus
$$\rho\,g\circ u_n
\xrightharpoonup[\bar{\Omega}]{}\lambda ~\text{(say)}.$$ Take $\lambda_{int}
=\frac{\lambda}{\rho}\chi_{\Omega}$ and $\lambda_{bd} =
\lambda\chi_{\partial\Omega}$. Then by the lemma $\ref{lem1}$,
\begin{eqnarray}
\lim_{n\rightarrow \infty} \int_\Omega (g \circ u_n) \varphi dx =
\int_\Omega \varphi d\lambda_{int} - \int_{\partial\Omega}
\frac{\partial \varphi}{\partial {\bf{n}}_{L^{*}}}\frac{1}{{\bf{n}}\cdot {\bf{n}} A^T} d\lambda \label{eq40}
\end{eqnarray} and since $\rho\mu _n \xrightharpoonup[\bar{\Omega}]{} \tau$ in
$\mathfrak{M}(\bar{\Omega})$,
\begin{eqnarray}
\displaystyle{\lim_{n\rightarrow \infty} \int_{\Omega}\varphi d
\mu_n=\int_{\Omega} \varphi
d\mu_{int}-\int_{\partial\Omega}\frac{\partial \varphi}{\partial
{\bf{n}}_{L^{*}}}\frac{1}{{\bf{n}}\cdot {\bf{n}} A^T} d\tau}\label{eq41}
\end{eqnarray}
for all $\varphi \in C^{2,L}_c(\bar{\Omega})$. Since $u_n$ is a weak
solution of $(\ref{ineq1})$, we have,
\begin{eqnarray*}
\int_\Omega (-u_n L^{*} \varphi + (g\circ u_n)\varphi)dx =
\int_\Omega \varphi d\mu_n - \int_{\partial\Omega}\frac{\partial
\varphi}{\partial {\bf{n}}_{L^{*}}} d\nu_n
\end{eqnarray*}
for every $\varphi \in C^{2,L}_c(\bar{\Omega})$. Taking the limit $n
\rightarrow \infty$ and using $(\ref{eq40})$ and $(\ref{eq41})$, we
have
\begin{align*}
-\int_\Omega u L^*\varphi dx + \int_\Omega \varphi d\lambda_{int} -
\int_{\partial\Omega} \frac{\partial \varphi}{\partial {\bf{n}}_{L^{*}}}
\frac{1}{{\bf{n}}\cdot {\bf{n}} A^T}d\lambda_{bd} = \int_{\Omega} \varphi d\mu_{int} &
-\int_{\partial\Omega}\frac{\partial
\varphi}{\partial {\bf{n}}_{L^{*}}}\frac{1}{{\bf{n}}\cdot {\bf{n}} A^T} d\tau_{bd}\\
& - \int_{\partial\Omega}\frac{\partial \varphi}{\partial
{\bf{n}}_{L^{*}}}d\nu
\end{align*} for every $\varphi \in C^{2,L}_c(\bar{\Omega})$. The above
equation can  also be expressed as
\begin{align*}
-\int_\Omega u L^*\varphi dx + \int_\Omega (g\circ u) \varphi dx =
\int_\Omega & (g\circ u) \varphi dx  - \int_{\Omega} \varphi
d(\lambda_{int} -\mu_{int})\\
& + \int_{\partial\Omega} \frac{\partial \varphi}{\partial
{\bf{n}}_{L^{*}}}\frac{1}{{\bf{n}}\cdot {\bf{n}} A^T} d(\lambda_{bd}-\tau_{bd})-
\int_{\partial\Omega}\frac{\partial \varphi}{\partial {\bf{n}}_{L^{*}}}d\nu
\end{align*}
for every $\varphi \in C_c^{2,L}(\bar{\Omega})$. This shows that $u$
is a weak solution of $(\ref{eq39})$, where
\begin{align}
\mu^{\#} & = g\circ u - (\lambda_{int} -\mu_{int})\,, \label{eq42}\\
\nu^{\#} & = \nu - \frac{(\lambda_{bd} - \tau_{bd})}{{\bf{n}}\cdot {\bf{n}}A^T}. \label{eq43}
\end{align}
Further, if $\mu_n, \nu_n\geq 0$ then by comparison of solutions
$u_n\geq 0$. Hence $\rho\,g\circ u_n\geq 0$ and in this case
$\lambda \geq 0$ and $\nu^{\#}\geq 0$. Also by uniformly ellipticity condition ($\ref{ellip}$), ${\bf{n}} \cdot {\bf{n}}A^T>0$. Hence by $(\ref{eq43})$, we
obtain $\displaystyle{\nu^{\#}\leq \nu +\frac{\tau_{bd}}{{\bf{n}}\cdot {\bf{n}} A^T}}$.
\end{proof}
\begin{remark}
 The Theorem $\ref{thm3}$ in this paper, is a generalization of the Theorem 4.1 of Bhakta and Marcus \cite{bhakta}, in which the case $L=-\Delta$ has been considered. In fact by putting $A=I$ in ($\ref{eq42}$) and ($\ref{eq43}$), we have the corresponding reduced limit   
 \begin{align*}
 \mu^{\#} & = g\circ u - (\lambda_{int} -\mu_{int})\,,\\
 \nu^{\#} & = \nu - (\lambda_{bd} - \tau_{bd}).
 \end{align*} 
 One more important thing is that the reduced limit of the boundary value problem depends on the matrix $A_{N\times N}$ corresponding to the principle part of the elliptic operator $L$.
\end{remark}
We now investigate the relation between the reduced limit and weak
limit which is given in terms of  the following theorem.
\begin{theorem}{\label{thm5}}
In addition to the assumptions of Theorem $\ref{thm3}$, assume that the uniform ellipticity condition ($\ref{ellip}$) holds with $\alpha \geq 1$ and also assume that the nonlinear function $g$-satisfies
\begin{eqnarray}
\lim_{a,t\rightarrow\infty}\frac{g(x,at)}{ag(x,t)}=\infty
\label{geq}
\end{eqnarray}
uniformly with respect to $x\in \Omega$. Let $v_n$ be the very weak solution
of
\begin{align}
\begin{split}
-L v_n & = \mu_n\,\,\,\mbox{in}\,\,\Omega,\\
v_n & = \nu_n \,\,\mbox{on}\,\,\partial\Omega \label{eq44}
\end{split}
\end{align}
If $\mu_n, \nu_n \geq 0$ and $\{g\circ v_n\}$ is bounded in
$L^1(\Omega; \rho)$ then $\nu^{\#}$ (reduced limit of $\{\nu_n\}$)
and $\displaystyle{\nu^{\#}\leq \nu +\frac{\tau_{bd}}{{\bf{n}}\cdot {\bf{n}} A^T}}$ are mutually absolutely continuous.
\end{theorem}
\begin{proof}
Since $\mu_n, \nu_n \geq 0$, hence by the theorem $\ref{thm3}$,
$0\leq \displaystyle{\nu^{\#}\leq \nu +\frac{\tau_{bd}}{{\bf{n}}\cdot {\bf{n}} A^T}}$. Therefore, $\nu^{\#}$ is
absolutely continuous with respect to $\displaystyle{\nu +\frac{\tau_{bd}}{{\bf{n}}\cdot {\bf{n}} A^T}}$. Thus we only
need to show $\displaystyle{\nu +\frac{\tau_{bd}}{{\bf{n}}\cdot {\bf{n}} A^T}}$ is absolutely continuous with respect
to $\nu^{\#}$.\\
Let $\alpha\in(0,1]$. Then we have $0\leq g\circ(\alpha v_n)\leq
g\circ v_n.$ By our assumption $\{g\circ v_n\}$ is bounded in
$L^1(\Omega; \rho)$. Hence there exists $c_0>0$ such that $$||
g\circ(\alpha v_n)||_{L^1(\Omega,\rho)}\leq c_0;\,\,\forall n\geq
1,\,\forall \alpha\in(0,1).$$ Let $\{\alpha_k\}$ be a sequence in
$(0, 1)$ such that $\alpha_k\downarrow 0.$ Then one can extract a
subsequence of $\{\rho \,g\circ (\alpha_k v_n)\}$ such that there
exists a measure $\sigma_k\in \mathfrak{M}(\bar{\Omega})$ such that
$$\rho g\circ (\alpha_k v_n)
\xrightharpoonup[\bar{\Omega}]{}\sigma_k$$ for each $k.$ Let
$w_{n,k}$ be the very weak solution of the problem
\begin{align}
\begin{split}
 -L w+g\circ w & =\alpha_k \mu_n \,\,\mbox{in}\,\,\Omega, \\
w & = \alpha_k \nu_n\,\,\mbox{on}\,\,\partial\Omega.\label{eq45}
\end{split}
\end{align}
We will denote $w_n$ to be the very weak solution of
\begin{align}
\begin{split}
 -L w+g\circ w & = \mu_n \,\,\mbox{in}\,\,\Omega, \\
w & = \nu_n\,\,\mbox{on}\,\,\partial\Omega.\label{eq46}
\end{split}
\end{align}
Since $ \alpha_k \mu_n \leq \mu_n$ and $\alpha_k \nu_n \leq \nu_n$,
hence by comparison of solutions when applied to ($\ref{eq45}$) and
($\ref{eq46}$), we have, $w_{n, k}\leq w_n$. Now observe that $g\circ
\alpha_k v_n \geq 0$. Since $v_n$ is a solution of ($\ref{eq44}$), we
have
\begin{align*}
-\int_\Omega \alpha_k v_n L^{*} \varphi dx + \int_\Omega(g\circ
\alpha_k v_n) \varphi dx & \geq -\int_\Omega v_n L^{*} \varphi dx
+ \int_\Omega(g\circ \alpha_k v_n) \varphi dx\\
& =  \int_\Omega \varphi d\mu_n - \int_{\partial\Omega}
\frac{\partial \varphi}{\partial {\bf{n}}_{L^{*}}} d\nu_n +
\int_\Omega(g\circ \alpha_k v_n) \varphi dx\\
& \geq \int_\Omega \varphi d\mu_n - \int_{\partial\Omega}
\frac{\partial \varphi}{\partial {\bf{n}}_{L^{*}}} d\nu_n
\end{align*}
for every $\varphi \geq 0 \in C^{2,L}_c(\bar{\Omega})$. This shows
that $\alpha_k v_n$ is a super solution of the problem ($\ref{eq46}$)
and hence $w_n \leq \alpha_k v_n$. Since $w_{n, k}\leq w_n$, we
obtain,
$$0\leq w_{n,k}\leq \alpha_k v_n.$$
As $\alpha_k v_n \leq v_n$ and $\{v_n\}$ is bounded in
$L^1(\Omega)$, hence there exists a subsequence of $\{w_{n, k}\}$
which converges in $L^1(\Omega)$, for each $k\in \mathbb{N}$. The
subsequence is still denoted by $\{w_{n, k}\}$. By the previous
theorem, $\{\rho (g\circ w_{n,k})\}$ converges weakly in
$\bar{\Omega}$ for each $k$; we denote its limit by $\lambda_k$. Let
$(\mu_k^{\#}, \nu_k ^{\#})$ be the reduced limit of $\{\alpha_k
\mu_n, \alpha_k \nu_n\}$. Again by the previous theorem, $$\displaystyle{\nu_k^{\#}
= \alpha_k \nu -\frac{(\lambda_k -\alpha_k \tau)}{{\bf{n}}\cdot {\bf{n}} A^T} \chi_{\partial\Omega}}.$$
As $w_{n, k} \leq \alpha_k v_n$ , hence
$$\rho(g\circ \alpha_k v_n) - \rho(g\circ
w_{n,k}) \xrightharpoonup[\bar{\Omega}]{}\sigma_k -\lambda_k \geq
0.$$ Now by our assumption, since the uniformly ellipticity condition ($\ref{ellip}$) holds with $\alpha\geq 1$, hence we have $\displaystyle{\sigma_k - \frac{\lambda_k}{{\bf{n}}\cdot {\bf{n}}A^T} \geq \sigma_k - \lambda_k \geq 0}$ in $\bar{\Omega}$. Thus we obtain,
\begin{eqnarray}
(\sigma_k - \frac{\lambda_k}{{\bf{n}}\cdot {\bf{n}}A^T})\chi_{\partial\Omega} = \sigma_k
\chi_{\partial\Omega} + \nu_k^{\#} - \alpha_k (\nu + \frac{\tau}{{\bf{n}} \cdot {\bf{n}}A^T}
\chi_{\partial\Omega}) \geq 0. \label{eq47}
\end{eqnarray}
Let $u_n$ be the solution of $(\ref{ineq1})$ corresponding to
$\mu=\mu_n$, $\nu=\nu_n$. By the comparison of solutions we have
$w_{n,k} \leq u_n$ for all $k, n\in \mathbb{N}$. Consequently,
$$w_k = \lim w_{n, k} \leq \lim u_n = u.$$ This implies that
\begin{eqnarray}
\nu_k ^{\#} = tr\, w_k \leq tr\, u \leq \nu^{\#}.\label{eq48}
\end{eqnarray}
Finally, from ($\ref{eq47}$) and ($\ref{eq48}$), we get
\begin{eqnarray}
\alpha_k (\nu + \frac{\tau}{{\bf{n}} \cdot {\bf{n}}A^T} \chi_{\partial\Omega}) \leq \sigma_k
\chi_{\partial\Omega} + \nu^{\#}.\label{eq49}
\end{eqnarray}
Since $g$ satisfies $(\ref{geq})$, hence for every $\epsilon >0$
there exists $a_0, t_0 >1$, such that
\begin{eqnarray}
\frac{g(x, at)}{a\,g(x, t)}\geq \frac{1}{\epsilon}\,, \,\,\,
\forall\, a\geq a_0,\, t\geq t_0. \label{eq50}
\end{eqnarray}
We split $\rho(g\circ \alpha_k v_n)$ as follows:
$$\rho(g\circ \alpha_k v_n) = \rho(g\circ \alpha_k v_n)\chi_{[\alpha_k
v_n <t_0]} + \rho(g\circ \alpha_k v_n) \chi_{[\alpha_k v_n \geq
t_0]}.$$ Now as $\rho(g\circ \alpha_k v_n)
\xrightharpoonup[\bar{\Omega}]{}\sigma_k$ , hence let us say
$$ \rho(g\circ \alpha_k v_n)\chi_{[\alpha_k
v_n <t_0]} \xrightharpoonup[\bar{\Omega}]{}\sigma_{1,k}\,\,\,\,
\mbox{and}\,\, \rho(g\circ \alpha_k v_n) \chi_{[\alpha_k v_n \geq
t_0]} \xrightharpoonup[\bar{\Omega}]{}\sigma_{2,k}.$$
 Since $\{ \rho(g\circ \alpha_k v_n)\chi_{[\alpha_k
v_n <t_0]}\}$ is uniformly bounded by $\rho (g\circ t_0)$, we have
$\sigma_{1, k} \chi_{\partial\Omega} =0 $. Thus $\sigma_k
\chi_{\partial\Omega}= \sigma_{2, k} \chi_{\partial\Omega}$. But
$$ ||\sigma_{2, k}||_{\mathfrak{M}(\bar{\Omega})} \leq \lim \inf
\int_{[\alpha_k v_n \geq t_0]} \rho (g \circ \alpha_k v_n).$$
Therefore we obtain,
$$||\sigma_k \chi_{\partial\Omega}||_{\mathfrak{M}(\partial\Omega)} \leq \lim \inf
\int_{[\alpha_k v_n \geq t_0]} \rho (g \circ \alpha_k v_n).$$ Now as
$\alpha_k \downarrow 0$, hence for sufficiently large $k$, say $k
\geq k_\epsilon$,  $\frac{1}{\alpha_k}\geq a_0$ we apply
$(\ref{eq50})$ with $a= \frac{1}{\alpha_k}$, $t= \alpha_k v_n$ to
get
$$\rho(g\circ \alpha_k v_n) \chi_{[\alpha_k v_n \geq t_0]} \leq
\alpha_k \epsilon (g \circ v_n),$$ for $k\geq k_\epsilon$ and $n\geq
1$. Hence
$$||\sigma_k \chi_{\partial\Omega}||_{\mathfrak{M}(\partial\Omega)} \leq
 \epsilon\alpha_k \lim \inf \int_\Omega \rho(g\circ v_n) \leq c_0 \epsilon
\alpha_k $$ for all $k\geq k_\epsilon$. Therefore
\begin{eqnarray}
\frac{||\sigma_k
\chi_{\partial\Omega}||_{\mathfrak{M}(\partial\Omega)}}{\alpha_k}\rightarrow
0\,,\,\,\, \mbox{as}\,\, k\rightarrow \infty.\label{eq51}
\end{eqnarray}
 To complete the
proof we will show that $\nu +\tau \chi_{\partial\Omega}$ is
absolutely continuous with respect to measure $\nu^{\#}$. For this
let $E\subset
\partial \Omega$ be a Borel set such that $\nu^{\#}(E) =0$. Then by
$(\ref{eq49})$,
$$\alpha_k(\nu(E) + \frac{\tau}{{\bf{n}}\cdot {\bf{n}}A^T}(E)) \leq \sigma_k(E)\,,\,\, \forall\,k\geq
1$$ This inequality and $(\ref{eq51})$ implies that
$$\nu(E) +\frac{\tau}{{\bf{n}}\cdot {\bf{n}}A^T}(E) \leq \frac{\sigma_k(E)}{\alpha_k} \leq \frac{|\sigma_k
\chi_{\partial\Omega}|(E)}{\alpha_k} \rightarrow 0$$ as
$k\rightarrow \infty$. Thus $\displaystyle{\nu(E) +\frac{\tau}{{\bf{n}}\cdot {\bf{n}}A^T}(E) =0}$. Hence the theorem.
\end{proof}
\section{Conclusions} The semilinear elliptic boundary value problem involving the general linear second order elliptic operator with
a nonlinear function and Radon measures has been studied. Although the existence of very weak solution may fail for general measure data input, we however proved that the boundary value problem considered here with $L^1$ data possesses a unique very weak solution.  We investigated the so-called reduced limits of the sequences $\{\mu_n,\nu_n\}$ of measures for a general linear elliptic operator $L$.  It is showed that the reduced limits strictly depends not only on the sequence of input measure datum but also on the elliptic differential operator $L$. We also gave the relation between the weak limit and the reduced limits of sequences of the given measures.
\bibliographystyle{amsplain}

\end{document}